\documentclass[11pt]{amsart}

\usepackage[margin=1.15in]{geometry}
\usepackage{amsmath,amssymb,amsthm,mathtools}
\usepackage{enumitem}
\usepackage{xcolor}
\usepackage[hidelinks]{hyperref}
\usepackage[nameinlink,capitalize]{cleveref}

\newtheorem{theorem}{Theorem}[section]
\newtheorem{lemma}[theorem]{Lemma}

\theoremstyle{definition}

\newtheorem{remark}[theorem]{Remark}

\newcommand{\Z}{\mathbb{Z}}

\title[A sharp $5/8$ bound for Erd\H{o}s 865]{A sharp $5/8$ bound for an Erd\H{o}s--S\'{o}s pairwise-sums problem}

\author{Ricky Cipollini}
\thanks{This manuscript was written by GPT-5.5 Pro from a proof developed by Ricky Cipollini together with GPT-5.5 Pro. The associated Lean formalization was carried out with Aristotle.}

\date{June 2026}

\begin{document}

\begin{abstract}
Let $f_3(N)$ be the least integer such that every set $A\subseteq \{1,\ldots,N\}$ of size at least $f_3(N)$ contains distinct $a,b,c\in A$ with
\[
  a+b,\quad a+c,\quad b+c\in A.
\]
We prove
\[
  f_3(N)\leq \frac58 N+O(1).
\]
Together with the standard construction $[N/8,N/4]\cup [N/2,N]$, this gives $f_3(N)=5N/8+O(1)$ and in particular resolves Erd\H{o}s Problem 865. The proof is self-contained. An earlier conditional version of the reduction has also been formalized in Lean 4/Mathlib with no sorries and no added axioms.
\end{abstract}

\maketitle

\section{Introduction}

For a positive integer $N$, let $[1,N]$ denote the set of integers $\{1,\ldots,N\}$. We say that $A\subseteq [1,N]$ contains a \emph{pairwise-sum triple} if there are distinct $a,b,c\in A$ such that
\[
  a+b,\quad a+c,\quad b+c \in A.
\]
Let $f_3(N)$ be the least integer such that every $A\subseteq[1,N]$ with $|A|\geq f_3(N)$ contains such a triple.

The problem considered here is the $k=3$ case of an Erd\H{o}s--S\'{o}s conjecture on pairwise sums. The example
\[
  A=\left[\frac N8,\frac N4\right]\cup\left[\frac N2,N\right]
\]
shows that the constant $5/8$ is best possible up to lower-order terms. The goal is to prove the matching upper bound.

The proof is self-contained. We prove the following theorem.

\begin{theorem}\label{thm:main}
There is a constant $C>0$ such that for all $N$, every $A\subseteq[1,N]$ with
\[
  |A|\geq \frac58N+C
\]
contains a pairwise-sum triple.
\end{theorem}

Equivalently, every pairwise-sum-triple-free set $A\subseteq[1,N]$ satisfies
\[
  |A|\leq \frac58N+O(1).
\]
All implicit constants in the proof are absolute.

\begin{remark}[Formalization]
An earlier version of the reduction from a coarse $2/3$ theorem to \Cref{thm:main} has been formalized in Lean 4/Mathlib with no sorries and no added axioms. The formalization proves a slightly weaker version of the folded additive lemma, the folding reduction, the centering argument, and the sharpness construction. It is currently outdated for the present version of the paper, since Lemma~\ref{lem:folded} has been sharpened and \Cref{sec:upper-bound} below replaces the coarse theorem by an induction argument.
\end{remark}

\section{The folded additive lemma}

We first prove a finite additive lemma on a folded interval. Let $m\geq 2$ and let
\[
  B\subseteq \{1,\ldots,m-1\}.
\]
Assume that for all distinct $x,y\in B$,
\begin{equation}\label{eq:foldedOK}
  x+y\neq m,
  \qquad
  x+y\notin B\pmod m.
\end{equation}
Define
\[
  S_0(B)=\{x+y:x\neq y,\ x,y\in B,\ x+y<m\},
\]
\[
  S_1(B)=\{x+y-m:x\neq y,\ x,y\in B,\ x+y>m\},
\]
and
\[
  C(B)=S_0(B)\cap S_1(B).
\]
Thus $C(B)$ is the set of residues which occur both as a non-wrapped and as a wrapped pair-sum of two distinct elements of $B$.

\begin{lemma}[Folded additive lemma]\label{lem:folded}
For every $m$ and $B$ satisfying \eqref{eq:foldedOK},
\[
  |B|-|C(B)|\leq \frac m4+2.
\]
\end{lemma}

\begin{proof}
We induct on $|B|$.

First observe the reflection symmetry. Put
\[
  -B=\{m-b:b\in B\}\subseteq \{1,\ldots,m-1\}.
\]
Replacing $B$ by $-B$ preserves \eqref{eq:foldedOK}; low sums for $B$ become high sums for $-B$, and high sums for $B$ become low sums for $-B$. Hence
\begin{equation}\label{eq:collision-reflection}
  |C(-B)|=|C(B)|.
\end{equation}

Assume $|B|\geq 2$. After replacing $B$ by $-B$ if necessary, we may suppose that the two smallest elements of $B$,
\[
  \alpha<\beta,
\]
satisfy
\begin{equation}\label{eq:small-sum}
  \alpha+\beta<m.
\end{equation}
Indeed, if the two smallest elements of $B$ have sum at least $m$, then the two largest elements $v<u$ of $B$ have $u+v>m$, since equality would give a forbidden sum equal to $m$. Reflecting sends them to the two smallest elements $m-u<m-v$ of $-B$, whose sum is $2m-u-v<m$.

We now split into two cases.

\smallskip
\noindent\textbf{Case 1: $\alpha+\beta\in C(B)$.}
Let
\[
  B'=B\setminus\{\alpha\}.
\]
Since $\alpha$ and $\beta$ are the two smallest elements of $B$, every low representation of $\alpha+\beta$ uses $\alpha$. Therefore, after deleting $\alpha$, the residue $\alpha+\beta$ no longer lies in $C(B')$. Deleting elements cannot create new collisions, so
\[
  |C(B')|\leq |C(B)|-1.
\]
Consequently
\[
  |B|-|C(B)|\leq |B'|-|C(B')|.
\]
The induction hypothesis applied to $B'$ gives the desired bound.

\smallskip
\noindent\textbf{Case 2: $\alpha+\beta\notin C(B)$.}
Work in the cyclic group $\Z/m\Z$. Consider the four sets
\[
  T_1=B,
  \qquad
  T_2=-B,
  \qquad
  T_3=(B-\alpha)\setminus\{0\},
  \qquad
  T_4=(\beta-B)\setminus\{0\},
\]
where these are viewed as subsets of $\Z/m\Z$. Their total size is
\begin{equation}\label{eq:four-total}
  |T_1|+|T_2|+|T_3|+|T_4|=4|B|-2.
\end{equation}
We record the required intersection bounds. First,
\[
  |T_1\cap T_2|\leq 1,
\]
and if $m$ is odd then $T_1\cap T_2=\varnothing$, since an element in this intersection gives two elements of $B$ whose sum is $m$, except possibly the single residue $m/2$. Also
\[
  |T_1\cap T_3|\leq 1,
\]
since $x\in T_1\cap T_3$ gives $x\in B$ and $x\equiv b-\alpha$ for some $b\in B$, so
\[
  \alpha+x\equiv b\pmod m;
\]
unless $x=\alpha$, this is a forbidden sum of two distinct elements of $B$ landing back in $B$. Similarly,
\[
  |T_1\cap T_4|\leq 2,
\]
since such an intersection gives $x+b\equiv \beta\pmod m$ with $x,b\in B$; unless $x=b$, this is forbidden, while $x=b$ forces $2x\equiv \beta\pmod m$, which has at most two solutions modulo $m$.

Next,
\[
  |T_2\cap T_3|\leq 1.
\]
Indeed, such an intersection gives $u+b\equiv \alpha\pmod m$ with $u,b\in B$. Unless $u=b$, this is forbidden, so $2u\equiv\alpha\pmod m$. This congruence has at most one solution in $B$: if $m$ is odd there is only one residue solution, while if $m$ is even then either there is no solution, or the two residue solutions are $\alpha/2$ and $\alpha/2+m/2$; the first is smaller than $\alpha$, so at most the second can lie in $B$.
Finally,
\[
  |T_2\cap T_4|\leq 1;
\]
indeed, such an intersection gives a congruence $u+\beta\equiv b\pmod m$ with $u,b\in B$, which is forbidden unless $u=\beta$.

It remains to bound $T_3\cap T_4$. Suppose
\[
  b-\alpha\equiv \beta-b'\pmod m
\]
with $b,b'\in B$. Then
\[
  b+b'\equiv \alpha+\beta\pmod m.
\]
Since $0<\alpha+\beta<m$, this means either
\[
  b+b'=\alpha+\beta
\]
or
\[
  b+b'=m+\alpha+\beta.
\]
The zero omissions in the definitions of $T_3$ and $T_4$ mean that $b\neq\alpha$ and $b'\neq\beta$. Thus in the first equality, since $\alpha,\beta$ are the two smallest elements of $B$, the only possible ordered contribution is $(b,b')=(\beta,\alpha)$. In the second equality, if $b\neq b'$, then $\alpha+\beta$ has both the low representation $\alpha+\beta$ and the high representation $b+b'-m$, contrary to $\alpha+\beta\notin C(B)$. If $b=b'$, then $2b=m+\alpha+\beta$, so there is at most one further possible residue. Hence
\[
  |T_3\cap T_4|\leq 2.
\]

Thus the total pairwise overlap is at most $8$ if $m$ is even and at most $7$ if $m$ is odd. Also $0$ lies in none of the sets $T_i$, so
\[
  \left|\bigcup_{i=1}^4T_i\right|\leq m-1.
\]
By the elementary union bound in the form
\[
  \sum_{i=1}^4 |T_i|
  \leq \left|\bigcup_{i=1}^4T_i\right|+
  \sum_{1\leq i<j\leq 4}|T_i\cap T_j|,
\]
we get $4|B|-2\leq m+7$ if $m$ is even and $4|B|-2\leq m+6$ if $m$ is odd. Thus $4|B|\leq m+9$ in the even case and $4|B|\leq m+8$ in the odd case. Since $4|B|$ is divisible by $4$, the even case also gives $4|B|\leq m+8$. Hence in all cases
\[
  |B|\leq \frac m4+2.
\]
Since $|C(B)|\geq 0$, this proves
\[
  |B|-|C(B)|\leq \frac m4+2.
\]
Together with the induction case, this proves the lemma.
\end{proof}

\section{Folding a triple-free set around a pivot}

Call a set $A\subseteq[1,N]$ \emph{triple-free} if it contains no pairwise-sum triple. Let $A$ be triple-free and let $h\in A$. Define
\[
  X=\{r:1\leq r<h,\ r\in A\},
\]
\[
  Y=\{r:1\leq r<h,\ h+r\leq N,\ h+r\in A\}.
\]
Put
\[
  B_h=X\cap Y,
  \qquad
  E=[1,h-1]\setminus(X\cup Y).
\]
Thus $B_h\subseteq\{1,\ldots,h-1\}$, and $r\in B_h$ means that both $r$ and $h+r$ lie in $A$.

\begin{lemma}[Folding lemma]\label{lem:folding}
One has
\[
  |X|+|Y|\leq \frac54h-|E\setminus C(B_h)|+1.
\]
In particular,
\[
  |X|+|Y|\leq \frac54h+1.
\]
\end{lemma}

\begin{proof}
For $h=1$ there is nothing to prove, so assume $h\geq 2$. Take distinct $x,y\in B_h$. Then
\[
  x,y,h,x+h,y+h\in A.
\]
If $x+y=h$, then $x,y,h$ form a forbidden triple. If $x+y\equiv r\pmod h$ for some $r\in B_h$, then either $x+y=r$ or $x+y=h+r$; in either case $x+y\in A$, and again $x,y,h$ form a forbidden triple. Hence $B_h$ satisfies the hypotheses of Lemma~\ref{lem:folded} modulo $h$.

We next show that
\begin{equation}\label{eq:collisions-empty}
  C(B_h)\subseteq E.
\end{equation}
If $r\in C(B_h)$, then $r$ has a low representation $r=x+y$ with distinct $x,y\in B_h$ and a high representation $h+r=u+v$ with distinct $u,v\in B_h$. If $r\in X$, then $r\in A$, and $x,y,h$ form a forbidden triple. Thus $r\notin X$. If $r\in Y$, then $h+r\in A$, and $u,v,h$ form a forbidden triple. Thus $r\notin Y$. Therefore $r\in E$.

Now
\[
  |X|+|Y|=|X\cup Y|+|X\cap Y|=(h-1-|E|)+|B_h|.
\]
Using \eqref{eq:collisions-empty} and Lemma~\ref{lem:folded},
\[
  |X|+|Y|=h-1+|B_h|-|C(B_h)|-|E\setminus C(B_h)|
  \leq \frac54h-|E\setminus C(B_h)|+1.
\]
The final assertion follows by discarding the nonnegative term $|E\setminus C(B_h)|$.
\end{proof}

\section{The upper bound}\label{sec:upper-bound}

It suffices to prove the theorem for even $N$, since for odd $N$ one may embed $[1,N]$ into $[1,N+1]$, changing the desired bound only by $O(1)$. Thus let $N=2H$. We prove, by strong induction on $H$, that every triple-free $A\subseteq[1,2H]$ satisfies
\begin{equation}\label{eq:even-target}
  |A|\leq \frac54H+6.
\end{equation}
The case $H=1$ is trivial. Assume $H\geq 2$ and that \eqref{eq:even-target} is known for all smaller positive values of $H$. Let $A\subseteq[1,2H]$ be triple-free.

If $A\cap[H,2H]=\varnothing$ or $A\cap[1,H]=\varnothing$, then $|A|\leq H$, so we may assume both intersections are nonempty. Define
\[
  q=H+s=\min(A\cap[H,2H]),
\]
and
\[
  p=H-e=\max(A\cap[1,H]).
\]
Then $s,e\geq 0$, and the open interval $(p,q)$ contains no element of $A$.

\subsection{The case \texorpdfstring{$s\leq 4e$}{s <= 4e}}
Fold around
\[
  h=q=H+s.
\]
Let
\[
  I=\{r\in\Z:\max(p,N-h)<r<h\}.
\]
If $s\geq 1$, then
\[
  |I|=s+\min(s,e)-1,
\]
and in all cases
\[
  |I|\geq s+\min(s,e)-1.
\]
For each $\rho\in I$, one has $\rho>p$ and $\rho<h=q$, so $\rho\notin A$ and hence $\rho\notin X$. Also $\rho>N-h$, so $h+\rho>N$, and hence $\rho\notin Y$. Thus $I\subseteq E$.

Moreover, $I\cap C(B_h)=\varnothing$. Indeed, every element $b\in B_h$ satisfies $b\leq N-h=H-s$. Hence every high folded sum has residue at most
\[
  2(N-h)-h=H-3s,
\]
whereas every $\rho\in I$ satisfies $\rho>N-h=H-s$. Therefore no $\rho\in I$ lies in $S_1(B_h)$.

Thus $I\subseteq E\setminus C(B_h)$. Since $|E\setminus C(B_h)|\geq |I|$, Lemma~\ref{lem:folding} gives
\[
  |X|+|Y|\leq \frac54h-|I|+1.
\]
Since $h=q\geq H$, every element of $A$ is counted by $X$, by $Y$, or is one of $h$ and possibly $2h$; the latter can occur only if $s=0$. If $s=0$, then $I=\varnothing$, and the preceding bound gives $|A|\leq \frac54H+3$. If $s\geq 1$, then there is only the endpoint $h$, and the preceding estimates give
\[
  |A|\leq \frac54(H+s)-s-\min(s,e)+3
  =\frac54H+\frac s4-\min(s,e)+3.
\]
If $e\geq s$, the last additional term is at most $s/4-s\leq 0$. If $e<s$, then $s\leq 4e$ gives $s/4-e\leq 0$. Hence \eqref{eq:even-target} follows in this case.

\subsection{The case \texorpdfstring{$s>4e$}{s > 4e}}
Fold around
\[
  h=p=H-e.
\]
Let
\[
  g=q-p=s+e.
\]
If $g>h$, then $g>H-e$. Since $g=s+e>5e$, we have $e<g/5$, and hence $H<g+e<6g/5$. Thus $g>5H/6$. The gap $(p,q)$ is empty, so
\[
  |A|\leq p+(2H-q+1)=2H-g+1<\frac76H+1\leq \frac54H+6.
\]
Thus we may assume $g\leq h$. Put
\[
  t=g-1.
\]
For every $1\leq r\leq t$, the upper partner $h+r=p+r$ lies in the gap $(p,q)$ and therefore is not in $A$. Hence the folded coordinate set satisfies
\[
  B_h\cap[1,t]=\varnothing.
\]
Let
\[
  A_0=A\cap[1,t],
\]
and put
\[
  I=[1,t]\setminus A_0.
\]
Every element of $I$ lies in $E$, and since $B_h\cap[1,t]=\varnothing$, no $\rho\leq t$ can be a low sum of two distinct elements of $B_h$. Hence
\[
  I\subseteq E\setminus C(B_h).
\]
Since $|E\setminus C(B_h)|\geq |I|$, Lemma~\ref{lem:folding} gives
\[
  |X|+|Y|\leq \frac54h-|I|+1.
\]
The part of $A$ up to $2h$ is counted by $X$, by $Y$, or by the two possible endpoints $h$ and $2h$. The tail beyond $2h$ has length $2H-2h=2e$. Consequently,
\begin{equation}\label{eq:case2-bound}
  |A|\leq \frac54h+2e-|I|+3
  =\frac54H+\frac34e+3-|I|.
\end{equation}
If $e\leq 3$, then \eqref{eq:case2-bound} and $|I|\geq 0$ give
\[
  |A|\leq \frac54H+6.
\]
Thus assume $e\geq 4$.

Since $A_0$ is triple-free and $t\leq h-1\leq H-1$, the induction hypothesis applies to $A_0$ after embedding $[1,t]$ in the even interval $[1,2\lceil t/2\rceil]$. Hence
\[
  |A_0|\leq \frac54\left\lceil\frac t2\right\rceil+6\leq \frac58(t+1)+6,
\]
and therefore
\[
  |I|=t-|A_0|\geq \frac38t-\frac{53}{8}.
\]
Since $s>4e$, we have $s\geq 4e+1$, and hence
\[
  t=s+e-1\geq 5e.
\]
Substituting the lower bound for $|I|$ into \eqref{eq:case2-bound}, we obtain
\[
  |A|\leq \frac54H+6+\left(\frac{29}{8}+\frac34e-\frac38t\right)\leq \frac54H+6,
\]
because $t\geq 5e$ and $e\geq 4$. This proves \eqref{eq:even-target}.

Combining the two cases, the induction is complete. By the initial reduction to even $N$, this proves \Cref{thm:main}.

\section{Sharpness}

For $N=8M$, set
\[
  A=[M,2M]\cup[4M,8M].
\]
Then
\[
  |A|=(M+1)+(4M+1)=5M+2=\frac58N+2.
\]
This set is triple-free. Two elements from the upper interval have sum greater than $N$. Two distinct elements from the lower interval have sum strictly between $2M=N/4$ and $4M=N/2$, hence outside $A$. Therefore any three distinct elements of $A$ contain either two upper elements, whose sum is too large, or two lower elements, whose sum is outside $A$. Thus no pairwise-sum triple exists in $A$.

This shows that the constant $5/8$ in \Cref{thm:main} is best possible.
Stijn Cambie has suggested the sharper conjectural formula that the maximum triple-free size is
\[
\left\lfloor \frac58N+2\right\rfloor-\mathbf 1_{N\equiv 5\pmod 8}
\]
for all \(N\), apart from the small exceptional values \(N\in\{1,2,13,14,21\}\). He also noted that the expected lower-bound constructions are obtained from the example above and its residue-class restrictions.

\section*{Acknowledgements}
This note was drafted by GPT-5.5 Pro. The mathematical findings and proof strategy are due to Ricky Cipollini together with GPT-5.5 Pro. The author thanks Stijn Cambie (\url{https://www.erdosproblems.com/forum/user/StijnC}) for feedback and improvements to the paper. The author also thanks Aristotle for the Lean 4/Mathlib formalization of the reduction. The formalization is available at \url{https://github.com/mrricky22/erdos-865-lean}.

\end{document}